\documentclass[9pt,oneside,english]{amsart}
\usepackage[T1]{fontenc}
\usepackage[latin9]{inputenc}
\usepackage{geometry}
\usepackage{amsthm}
\usepackage{amsbsy}
\usepackage{amstext}
\usepackage{amssymb}

\makeatletter
\numberwithin{equation}{section}
\numberwithin{figure}{section}
  \theoremstyle{plain}
  \newtheorem*{thm*}{\protect\theoremname}
\theoremstyle{plain}
\newtheorem{thm}{\protect\theoremname}
  \theoremstyle{definition}
  \newtheorem{defn}[thm]{\protect\definitionname}
  \theoremstyle{plain}
  \newtheorem{prop}[thm]{\protect\propositionname}
  \theoremstyle{remark}
  \newtheorem{rem}[thm]{\protect\remarkname}
  \theoremstyle{plain}
  \newtheorem{cor}[thm]{\protect\corollaryname}
  \theoremstyle{definition}
  \newtheorem{example}[thm]{\protect\examplename}
  \theoremstyle{plain}
  \newtheorem{lem}[thm]{\protect\lemmaname}

\usepackage{xypic}
\usepackage{pstricks,pst-node,pst-tree}
\theoremstyle{definition}
\newtheorem{parn}{}[subsection]

\makeatother

\usepackage{babel}
  \providecommand{\corollaryname}{Corollary}
  \providecommand{\definitionname}{Definition}
  \providecommand{\examplename}{Example}
  \providecommand{\lemmaname}{Lemma}
  \providecommand{\propositionname}{Proposition}
  \providecommand{\remarkname}{Remark}
  \providecommand{\theoremname}{Theorem}
\providecommand{\theoremname}{Theorem}

\begin{document}

\author{Adrien Dubouloz}

\address{Institut de Math\'ematiques de Bourgogne, 9 avenue Alain Savary,
21 078 Dijon, France}

\email{adrien.dubouloz@u-bourgogne.fr}

\thanks{This project was partialy funded by ANR Grant \textquotedbl{}BirPol\textquotedbl{}
ANR-11-JS01-004-01. }


\dedicatory{Dedicated to Wlodek Danielewski. }

\title{On the cancellation problem for algebraic tori}
\begin{abstract}
We address a variant of Zariski Cancellation Problem, asking whether
two varieties which become isomorphic after taking their product with
an algebraic torus are isomorphic themselves. Such cancellation property
is easily checked for curves, is known to hold for smooth varieties
of log-general type by virtue of a result of Iitaka-Fujita and more
generally for non $\mathbb{A}_{*}^{1}$-uniruled varieties. We show
in contrast that for smooth affine factorial $\mathbb{A}_{*}^{1}$-ruled
varieties, cancellation fails in any dimension bigger or equal to
two. 
\end{abstract}
\maketitle
Since the late seventies, the Cancellation Problem is usually understood
in its geometric form as the question whether two algebraic varieties
$X$ and $Y$ with isomorphic cylinders $X\times\mathbb{A}^{1}$ and
$Y\times\mathbb{A}^{1}$ are isomorphic themselves. This problem is
intimately related to the geometry of rational curves on $X$ and
$Y$: in particular, if $X$ or $Y$ are smooth quasiprojective and
not $\mathbb{A}^{1}$-uniruled, in the sense that they do not admit
any dominant generically finite morphism from a variety of the form
$Z\times\mathbb{A}^{1}$, then every isomorphism $\Phi:X\times\mathbb{A}^{1}\stackrel{\sim}{\rightarrow}Y\times\mathbb{A}^{1}$
descends to an isomorphism between $X$ and $Y$, a property which
is sometimes called strong cancellation. Over an algebraically closed
field of characteristic zero, the non $\mathbb{A}^{1}$-uniruledness
of a smooth quasi-projective variety $X$ is guaranteed in particular
by the existence of pluri-forms with logarithmic poles at infinity
on suitable projective completions of $X$, a property which can be
read off from the non-negativity of a numerical invariant of $X$,
called its (logarithmic) Kodaira dimension $\kappa(X)$, introduced
by S. Iitaka \cite{Ii77} as the analogue of the usual notion of Kodaira
dimension for complete varieties. In this setting, it was established
by S. Iitaka et T. Fujita \cite{IiFu77} that strong cancellation
does hold for a large class of smooth varieties, namely whenever $X$
or $Y$ has non-negative Kodaira dimension. This general result implies
in particular that cancellation holds for smooth affine curves, due
to the fact that the affine line $\mathbb{A}^{1}$ is the only such
curve with negative Kodaira dimension. 

All these assumptions turned out to be essential, as shown by a famous
unpublished counter-example due to W. Danielewski \cite{Dan89} of
a pair of non-isomorphic smooth complex $\mathbb{A}^{1}$-ruled affine
surfaces with isomorphic cylinders. The techniques introduced by W.
Danielewski have been the source of many progress on the Cancellation
Problem during the last decade but, except for the case of the affine
plane $\mathbb{A}^{2}$ which was solved earlier affirmatively by
M. Miyanishi and T. Sugie \cite{MiSu80}, the question whether cancellation
holds for the complex affine space $\mathbb{A}^{n}$ remains one of
the most challenging and widely open problem in affine algebraic geometry.
In contrast, the same question in positive characteristic was recently
settled by the negative by N. Gupta \cite{Gu14}, who checked using
algebraic methods developed by A. Crachiola and L. Makar-Limanov that
a three-dimensional candidate constructed by T. Asanuma \cite{As87}
was indeed a counter-example. \\

In this article, we consider another natural cancellation problem
in which $\mathbb{A}^{1}$ is replaced by the punctured affine line
$\mathbb{A}_{*}^{1}\simeq\mathrm{Spec}(\mathbb{C}[x^{\pm1}])$ or,
more generally, by an algebraic torus $\mathbb{T}^{n}=\mathrm{Spec}(\mathbb{C}[x_{1}^{\pm1},\ldots,x_{n}^{\pm1}])$,
$n\geq1$. The question is thus whether two, say smooth quasi-projective,
varieties $X$ and $Y$ such that $X\times\mathbb{T}^{n}$ is isomorphic
to $Y\times\mathbb{T}^{n}$ are isomorphic themselves. In contrast
with the usual Cancellation Problem, this version seems to have received
much less attention, one possible reason being that the analogue in
this context of the Cancellation Problem for $\mathbb{A}^{n}$, namely
the question whether an affine variety $X$ such that $X\times\mathbb{T}^{n}$
is isomorphic to $\mathbb{T}^{n+m}$ is itself isomorphic to the torus
$\mathbb{T}^{n}$, admits an elementary positive answer derived from
the knowledge of the structure of the automorphism groups of algebraic
tori: indeed, the action of the torus $\mathbb{T}^{n}=\mathrm{Spec}(\mathbb{C}[M'])$
by translations on the second factor of $X\times\mathbb{T}^{n}$ corresponds
to a grading of the algebra of the torus $\mathbb{T}^{n+m}=\mathrm{Spec}(\mathbb{C}[M])$
by the lattice $M'$ of characters of $\mathbb{T}^{n}$, induced by
a surjective homomorphism $\sigma:M\rightarrow M'$ from the lattice
of characters $M$ of $\mathbb{T}^{n+m}$. Letting $\tau:M'\rightarrow M$
be a section of $\sigma$, $M''=M/\tau(M')$ is a lattice of rank
$m$ for which we have isomorphisms of algebraic quotients $X\simeq X\times\mathbb{T}^{n}/\!/\mathbb{T}^{n}\simeq\mathbb{T}^{n+m}/\!/\mathbb{T}^{n}\simeq\mathbb{T}^{m}=\mathrm{Spec}(\mathbb{C}[M''])$.

Without such precise information on automorphism groups, the question
for general varieties $X$ and $Y$ is more complicated. Of course,
appropriate conditions on the structure of invertible functions on
$X$ and $Y$ can be imposed to guarantee that cancellation holds
(see \cite{Fr14} for a detailed discussion of this point of view):
this is the case for instance when either $X$ or $Y$ does not have
non constant such functions. Indeed, given an isomorphism $\Phi:X\times\mathbb{A}_{*}^{1}\stackrel{\sim}{\rightarrow}Y\times\mathbb{A}_{*}^{1}$,
the restriction to every closed fiber of the first projection $\mathrm{pr}_{1}:X\times\mathbb{A}_{*}^{1}\rightarrow X$
of the composition of $\Phi$ with the second projection $\mathrm{pr}_{2}:Y\times\mathbb{A}_{*}^{1}\rightarrow\mathbb{A}_{*}^{1}$
induces an invertible function on $X$, implying that $\Phi$ descends
to an isomorphism between $X$ and $Y$ as soon as every such function
on $X$ is constant. 

But from a geometric point of view, it seems that the cancellation
property for $\mathbb{A}_{*}^{1}$ is again related to the nature
of affine rational curves contained in the varieties $X$ and $Y$,
more specifically to the geometry of images of the punctured affine
line $\mathbb{A}_{*}^{1}$ on them. It is natural to expect that strong
cancellation holds for varieties which are not dominantly covered
by images of $\mathbb{A}_{*}^{1}$, but this property is harder to
characterize in terms of numerical invariants. In particular, in every
dimension $\geq2$, there exists smooth $\mathbb{A}_{*}^{1}$-uniruled
affine varieties $X$ of any Kodaira dimension $\kappa(X)\in\{-\infty,0,1,\ldots,\dim X-1\}$.
In contrast, a smooth complex affine variety $X$ of log-general type,
i.e. of maximal Kodaira dimension $\kappa(X)=\dim X$, is not $\mathbb{A}_{*}^{1}$-uniruled,
and another general result of I. Itaka and T. Fujita \cite{IiFu77}
does indeed confirm that strong cancellation holds for products of
algebraic tori with smooth affine varieties of log-general type. Combined
with the fact that $\mathbb{A}_{*}^{1}$ is the unique smooth affine
curve of Kodaira dimension $0$, this is enough for instance to conclude
that cancellation holds for smooth affine curves. 

Our main result, which can be summarized as follows, shows that similarly
as in the case of the usual Cancellation Problem for $\mathbb{A}^{1}$,
these assumptions are essential:
\begin{thm*}
In every dimension $d\geq2$, there exists non isomorphic smooth factorial
affine $\mathbb{A}_{*}^{1}$-uniruled varieties $X$ and $Y$ of dimension
$d$ and Kodaira dimension $d-1$ with isomorphic $\mathbb{A}_{*}^{1}$-cylinders
$X\times\mathbb{A}_{*}^{1}$ and $Y\times\mathbb{A}_{*}^{1}$. 
\end{thm*}
In dimension $d\geq3$, these families are obtained in the form of
total spaces of suitable Zariski locally trivial $\mathbb{A}_{*}^{1}$-bundles
over smooth affine varieties of log-general type. The construction
guarantees the isomorphy between the corresponding $\mathbb{A}_{*}^{1}$-cylinders
thanks to a fiber product argument reminiscent to the famous Danielewski
fiber product trick in the case of the usual Cancellation Problem.
The two-dimensional counter-examples are produced along the same lines,
at the cost of replacing the base varieties of the $\mathbb{A}_{*}^{1}$-bundles
involved in the construction by appropriate orbifold curves. The article
is organized as follows: in the first section, we establish a variant
of Iitaka-Fujita strong cancellation Theorem for Zariski locally trivial
$\mathbb{T}^{n}$-bundles over smooth affine varieties of log-general
type. This criterion is applied in the second section to deduce the
existence of families of Zariski locally trivial $\mathbb{A}_{*}^{1}$-bundles
over smooth affine varieties of log-general type with non isomorphic
total spaces but isomorphic $\mathbb{A}_{*}^{1}$-cylinders. The two-dimensional
case is treated in a separate sub-section. The last section contains
a generalization of some of these constructions to the cancellation
problem for higher dimensional tori $\mathbb{T}^{n}$ over varieties
of dimension at least three, and a complete discussion of the cancellation
problem for $\mathbb{A}_{*}^{1}$ in the special case of smooth factorial
affine surfaces.

\section{A criterion for cancellation}

\subsection{Recollection on locally trivial $\mathbf{T}^{n}$-bundles }

In what follows, we denote by $\mathbf{T}^{n}$ the spectrum of the
Laurent polynomial algebra $\mathbb{C}[t_{1}^{\pm1},\ldots,t_{n}^{\pm1}]$
in $n$ variables. We use the notation $\mathbb{T}^{n}$ to indicate
that we consider $\mathbf{T}^{n}$ as the product $\mathbb{G}_{m}^{n}$,
i.e. $\mathbf{T}^{n}$ equipped with its natural algebraic group product
structure. The automorphism group $\mathrm{Aut}(\mathbf{T}^{n})$
of $\mathbf{T}^{n}$ is isomorphic to the semi-direct product $\mathbb{T}^{n}\rtimes\mathrm{GL}_{n}(\mathbb{Z})$,
where $\mathbb{T}^{n}$ acts on $\mathbf{T}^{n}$ by translations
and where $\mathrm{GL}_{n}(\mathbb{Z})$ acts by $(a_{ij})_{i,j=1,\ldots,n}\cdot(t_{1},\ldots,t_{n})=(\prod_{i=1}^{n}t_{i}^{a_{1i}},\ldots,\prod_{i=1}^{n}t_{i}^{a_{ni}})$. 
\begin{defn}
A \emph{Zariski locally trivial} $\mathbf{T}^{n}$-\emph{bundle} over
a scheme $X$, is an $X$-scheme $p:P\rightarrow X$ for which every
point of $X$ has a Zariski open neighbourhood $U\subset X$ such
that $p^{-1}\left(U\right)\simeq U\times\mathbf{T}^{n}$ as schemes
over $U$. 
\end{defn}
\begin{parn} \label{par:isoclass-exact-sequence} Isomorphy classes
of Zariski locally trivial $\mathbf{T}^{n}$-bundle over $X$ are
in one-to-one correspondence with elements of the \v{C}ech cohomology
group $\check{H}^{1}(X,\mathrm{Aut}(\mathbf{T}^{n}))$. Furthermore,
letting $\mathrm{GL}_{n}(\mathbb{Z})_{X}$ denote the locally constant
sheaf $\mathrm{GL}_{n}(\mathbb{Z})$ on $X$, we derive from the short
exact sequence $0\rightarrow\mathbb{T}^{n}\rightarrow\mathrm{Aut}(\mathbf{T}^{n})\rightarrow\mathrm{GL}_{n}(\mathbb{Z})_{X}\rightarrow0$
of sheaves over $X$ the following long exact sequence in \v{C}ech
cohomology 
\[
0\rightarrow\check{H}^{0}(X,\mathbb{T}^{n})\rightarrow\check{H}^{0}(X,\mathrm{Aut}(\mathbf{T}^{n}))\rightarrow\check{H}^{0}(X,\mathrm{GL}_{n}(\mathbb{Z})_{X})\rightarrow\check{H}^{1}(X,\mathbb{T}^{n})\rightarrow\check{H}^{1}(X,\mathrm{Aut}(\mathbf{T}^{n}))\rightarrow\check{H}^{1}(X,\mathrm{GL}_{n}(\mathbb{Z})_{X}).
\]
If $X$ is irreducible, then $\check{H}^{1}(X,\mathrm{GL}_{n}(\mathbb{Z})_{X})=0$
and so, every Zariski locally trivial $\mathbf{T}^{n}$-bundle can
be equipped with the additional structure of a principal homogeneous
$\mathbb{T}^{n}$-bundle. Moreover, two principal homogeneous $\mathbb{T}^{n}$-bundles
have isomorphic underlying $\mathbf{T}^{n}$-bundles if and only if
their isomorphy classes in $\check{H}^{1}(X,\mathbb{T}^{n})\simeq H^{1}(X,\mathbb{T}^{n})$
belong to the same orbit of the natural action of $\check{H}^{0}(X,\mathrm{GL}_{n}(\mathbb{Z})_{X})\simeq\mathrm{GL}_{n}(\mathbb{Z})$
which, for every $(a_{ij})_{i,j=1,\ldots n}\in\mathrm{GL}_{n}(\mathbb{Z})$,
sends the isomorphy class of the $\mathbb{T}^{n}$-bundle $p:P\rightarrow X$
with action $\mathbb{T}^{n}\times P\rightarrow P$, $\left((t_{1},\cdots,t_{n}),p\right)\mapsto(t_{1},\ldots,t_{n})\cdot p$
to the isomorphy class of $p:P\rightarrow X$ equipped with the action
$((t_{1},\ldots t_{n}),p)\mapsto((a_{ij})_{i,j=1,\ldots,n}\cdot(t_{1},\ldots,t_{n}))\cdot p$.
In other words, for an irreducible $X$, isomorphy classes of Zariski
locally trivial $\mathbf{T}^{n}$-bundles over $X$ are in one-to-one
correspondence with elements of $H^{1}(X,\mathbb{T}^{n})/\mathrm{GL}_{n}(\mathbb{Z})$. 

\end{parn}

\subsection{Cancellation for $\mathbf{T}^{n}$-bundles over varieties of log-general
type}

Recall that the \emph{(logarithmic) Kodaira dimension} $\kappa(X)$
of a smooth complex algebraic variety $X$ is the Iitaka dimension
of the invertible sheaf $\omega_{\overline{X}/\mathbb{C}}(\log B)=(\det\Omega_{\overline{X}/\mathbb{C}}^{1})\otimes\mathcal{O}_{\overline{X}}(B)$
on a smooth complete model $\overline{X}$ of $X$ with reduced SNC
boundary divisor $B=\overline{X}\setminus X$. So $\kappa(X)$ is
equal to $\mbox{\ensuremath{\mathrm{tr}}.\ensuremath{\deg}}_{\mathbb{C}}(\bigoplus_{m\geq0}H^{0}(\overline{X},\omega_{\overline{X}/\mathbb{C}}(\log B)^{\otimes m}))-1$
if $H^{0}(\overline{X},\omega_{\overline{X}/\mathbb{C}}(\log B)^{\otimes m})\neq0$
for sufficiently large $m$, and, by convention to $-\infty$ otherwise.
The so-defined element of $\{-\infty\}\cup\{0,\ldots,\mathrm{dim}_{\mathbb{C}}X\}$
is independent of the choice of a smooth complete model $(\overline{X},B)$
\cite{Ii77} and coincides with the usual Kodaira dimension in the
case where $X$ is complete. A smooth variety $X$ such that $\kappa(X)=\mathrm{dim}_{\mathbb{C}}X$
is said to be of log-general type. 

The following Proposition is a variant for Zariski locally trivial
bundles of Iitaka-Fujita's strong Cancellation Theorem \cite[Theorem 3]{IiFu77}
for products of varieties of log-general type with affine varieties
of Kodaira dimension equal to $0$, such as algebraic tori $\mathbf{T}^{n}$. 
\begin{prop}
\label{prop:Bundle-Cancel} Let $X$ and $Y$ be smooth algebraic
varieties and let $p:P\rightarrow X$ and $q:Q\rightarrow Y$ be Zariski
locally trivial $\boldsymbol{T}^{n}$-bundles. If either $X$ or $Y$
is of log-general type then for every isomorphism of abstract algebraic
varieties $\Phi:P\stackrel{\sim}{\rightarrow}Q$ between the total
spaces of $P$ and $Q$, there exists an isomorphism $\varphi:X\stackrel{\sim}{\rightarrow}Y$
such that the diagram 
\begin{eqnarray*}
P & \stackrel{\Phi}{\rightarrow} & Q\\
p\downarrow &  & \downarrow q\\
X & \stackrel{\varphi}{\rightarrow} & Y
\end{eqnarray*}
commutes. \end{prop}
\begin{proof}
The proof is very similar to that of \cite[Theorem 1]{IiFu77}. We
may assume without loss of generality that $Y$ is of log-general
type. Since $p$ has local sections in the Zariski topology, it is
enough to show that $q\circ\Phi$ is constant on the fibers of $p$
to guarantee that the induced set-theoretic map $\varphi:X\rightarrow Y$
is a morphism. Furthermore, since $\Phi$ is an isomorphism, $\varphi$
will be bijective whence an isomorphism by virtue of Zariski Main
Theorem \cite[8.12.6]{EGAIV}. Since $p:P\rightarrow X$ is Zariski
locally trivial and $\kappa(\mathbf{T}^{n})=0$, it follows from \cite{Ii77}
that for every prime Weil divisor $D$ on $X$, the Kodaira dimension
$\kappa(p^{-1}(D_{\mathrm{reg}}))$ of the inverse image of the regular
part of $D$ is at most equal to $\dim D$. This implies in turn that
the restriction of $q\circ\Phi$ to $p^{-1}(D)$ cannot be dominant
since otherwise we would have $\kappa(Y)\leq\kappa(p^{-1}(D_{\mathrm{reg}}))<\dim X=\dim Y$,
in contradiction with the assumption that $\kappa(Y)=\dim Y$. So
there exists a prime Weil divisor $D'$ on $Y$ such that the image
of $p^{-1}(D)$ by $\Phi$ is contained in $q^{-1}(D')$, whence is
equal to it since they are both irreducible of the same dimension.
Now given any closed point $x\in X$, we can find a finite collection
of prime Weil divisors $D_{1},\ldots,D_{n}$ such that $D_{1}\cap\cdots D_{n}=\left\{ x\right\} $.
Letting $D_{i}'$ be a collection of prime Weil divisors on $Y$ such
that $\Phi(p^{-1}(D_{i}))=q^{-1}(D_{i}')$ for every $i=1,\ldots,n$,
we have 
\[
q^{-1}(\bigcap_{i=1,\ldots,n}D_{i}')=\bigcap_{i=1,\ldots,n}q^{-1}(D_{i}')=\bigcap_{i=1,\ldots,n}\Phi(p^{-1}(D_{i}))\simeq\Phi(\bigcap_{i=1,\ldots,n}p^{-1}(D_{i}))\simeq\Phi(\{x\}\times\mathbf{T}^{n})\simeq\mathbf{T}^{n}.
\]
So the intersection of the $D_{i}'$, $i=1,\ldots,n$, consists of
a unique closed point $y\in Y$ for which we have by construction
$\Phi(p^{-1}(x))=q^{-1}(y)$, as desired.\end{proof}
\begin{rem}
The proof above shows in fact that the conclusion of the Theorem holds
under the more geometric hypothesis that either $X$ or $Y$ is not
$\mathbb{A}_{*}^{1}$-uniruled, i.e., does not admit any dominant
generically finite morphism from a variety of the form $Z\times\mathbb{A}_{*}^{1}$.
In particular strong cancellation holds for products of algebraic
tori $\mathbf{T}^{n}$ with non $\mathbb{A}_{*}^{1}$-uniruled varieties. \end{rem}
\begin{cor}
Two smooth curves $C$ and $C'$ admit Zariski locally trivial $\mathbf{T}^{n}$-bundles
$p:P\rightarrow C$ and $p':P'\rightarrow C'$ with isomorphic total
spaces $P$ and $P'$ if and only if they are isomorphic. \end{cor}
\begin{proof}
If either $C$ or $C'$ is of log-general type, then the assertion
follows from Proposition \ref{prop:Bundle-Cancel}. Note further that
$C$ is affine if and only if so is $P$. Indeed, $p:P\rightarrow C$
is an affine morphism and conversely, if $P$ is affine, then viewing
$P$ as a principal homogeneous $\mathbb{T}^{n}$ -bundle with geometric
quotient $P/\!/\mathbb{T}^{n}\simeq C$, the affineness of $C$ follows
from the fact that  the algebraic quotient morphism $P\rightarrow P/\!/\mathbb{T}=\mathrm{Spec}(\Gamma(P,\mathcal{O}_{P})^{\mathbb{T}})$
is a categorical quotient in the category of algebraic varieties,
so that $C\simeq\mathrm{Spec}(\Gamma(P,\mathcal{O}_{P})^{\mathbb{T}})$.
Thus $C$ and $C'$ are simultaneously affine or projective. In the
first case, $C$ and $C'$ are isomorphic to either the affine line
$\mathbb{A}^{1}$ or the punctured affine line $\mathbb{A}_{*}^{1}$
which both have a trivial Picard group. So $P$ and $P'$ are trivial
$\mathbf{T}^{n}$-bundles and the isomorphy of $C$ and $C'$ follows
by comparing invertible function on $P$ and $P'$. In the second
case, if either $C$ or $C'$ has non negative genus, say $g(C')\geq0$,
then, being rational, the image of a fiber of $p:P\rightarrow C$
by an isomorphism $\Phi:P\stackrel{\sim}{\rightarrow}P'$ must be
contained in a fiber of $p':P'\rightarrow C'$. We conclude similarly
as in the proof if the previous Proposition that $\Phi$ descends
to an isomorphism between $C$ and $C'$. 
\end{proof}
\begin{parn} The automorphism group $\mathrm{Aut}(X)$ of a scheme
$X$ acts on the set of isomorphy classes of principal homogeneous
$\mathbb{T}^{n}$-bundles over $X$ via the linear representation
\[
\eta:{\rm Aut}\left(X\right)\rightarrow{\rm GL}(H^{1}(X,\mathbb{T}^{n})),\;\psi\mapsto\eta\left(\psi\right)=\psi^{*}:H^{1}(X,\mathbb{T}^{n})\stackrel{\sim}{\rightarrow}H^{1}(X,\mathbb{T}^{n}),
\]
 where $\psi^{*}$ maps the isomorphy class of principal homogeneous
$\mathbb{T}^{n}$-bundle $p:P\rightarrow X$ to the one of the $\mathbb{T}^{n}$-bundle
${\rm pr}_{2}:P\times_{p,X,\psi}X\rightarrow X$. This action commutes
with natural action of $\mathrm{GL}_{n}(\mathbb{Z})$ introduced in
$\S$ \ref{par:isoclass-exact-sequence} above, and Proposition \ref{prop:Bundle-Cancel}
implies the following characterization:

\end{parn}
\begin{cor}
\label{cor:IsoClass-logGeneral} Over a smooth variety $X$ of log-general
type, the set $H^{1}(X,\mathbb{T}^{n})/(\mathrm{Aut}(X)\times\mathrm{GL}_{n}(\mathbb{Z}))$
parametrizes isomorphy classes as abstract varieties of total spaces
of Zariski locally trivial $\boldsymbol{T}^{n}$-bundles $p:P\rightarrow X$.
\end{cor}

\section{non-Cancellation for the $1$-dimensional torus \label{sec:non-Cancel}}

Candidates for non-cancellation of the $1$-dimensional torus $\mathbf{T}=\mathbb{A}_{*}^{1}=\mathrm{Spec}(\mathbb{C}[t^{\pm1}])$
can be constructed along the following lines: given say a smooth quasi-projective
variety $X$ and a pair of non-isomorphic principal homogeneous $\mathbb{G}_{m}$-bundles
$p:P\rightarrow X$ and $q:Q\rightarrow X$ whose classes generate
the same subgroup of $H^{1}(X,\mathbb{G}_{m})$, the fiber product
$W=P\times_{X}Q$ is a principal homogeneous $\mathbb{T}^{2}$-bundle
over $X$, which inherits the structure of a principal $\mathbb{G}_{m}$-bundle
over $P$ and $Q$ simultaneously, via the first and the second projection
respectively. Since the classes of $P$ and $Q$ generate the same
subgroup of $H^{1}(X,\mathbb{G}_{m})$, it follows that the classes
of $\mathrm{pr}_{1}:W\simeq p^{*}Q\rightarrow P$ and $\mathrm{pr}_{2}:W\simeq q^{*}P\rightarrow Q$
in $H^{1}(P,\mathbb{G}_{m})$ and $H^{1}(Q,\mathbb{G}_{m})$ respectively
are both trivial and so, we obtain isomorphisms $P\times\mathbb{G}_{m}\simeq W\simeq Q\times\mathbb{G}_{m}$
of locally trivial $\mathbf{T}^{2}$-bundles over $X$. 

Then we are left with finding appropriate choices of $X$ and classes
in $H^{1}(X,\mathbb{G}_{m})$ which guarantee that the total spaces
of the corresponding principal homogeneous $\mathbb{G}_{m}$-bundles
$p:P\rightarrow X$ and $q:Q\rightarrow X$ are not isomorphic as
abstract algebraic varieties.

\subsection{\label{sub:Torus-NonCancel-dim3}Non-cancellation for smooth factorial
affine varieties of dimension $\geq3$ }

A direct application of the above strategy leads to families of smooth
factorial affine varieties of any dimension $\geq3$ for which cancellation
fails: 
\begin{prop}
\label{thm:Non-Cancel-Gm} Let $X$ be the complement of a smooth
hypersurface $D$ of degree $d$ in $\mathbb{P}^{r}$, $r\geq2$,
such that $d\geq r+2$ and $|\left(\mathbb{Z}/d\mathbb{Z}\right)^{*}|\geq3$,
and let $p:P\rightarrow X$ and $q:Q\rightarrow X$ be the $\mathbb{G}_{m}$-bundles
corresponding to the line bundles $\mathcal{O}_{\mathbb{P}^{n}}(1)\mid_{X}$
and $\mathcal{O}_{\mathbb{P}^{n}}(k)\mid_{X}$, $k\in\left(\mathbb{Z}/d\mathbb{Z}\right)^{*}\setminus\{\overline{1},\overline{d-1}\}$
under the isomorphism $H^{1}(X,\mathbb{G}_{m})\simeq\mathrm{Pic}(X)$.
Then $P$ and $Q$ are not isomorphic as abstract algebraic varieties
but $P\times\mathbb{A}_{*}^{1}$ and $Q\times\mathbb{A}_{*}^{1}$
are isomorphic as schemes over $X$. \end{prop}
\begin{proof}
The Picard group of $X$ is isomorphic to the group $\mu_{d}\simeq\mathbb{Z}/d\mathbb{Z}$
of $d$-th roots of unity, generated by the restriction of $\mathcal{O}_{\mathbb{P}^{n}}(1)$
to $X$. Since $k$ is relatively prime with $d$, $\mathcal{O}_{\mathbb{P}^{r}}(k)\mid_{X}$
is also a generator of $\mathrm{Pic}(X)$. This guarantees that $P\times\mathbb{A}_{*}^{1}$
is isomorphic to $Q\times\mathbb{A}_{*}^{1}$ by virtue of the previous
discussion. Since $d\geq r+2$, $K_{\mathbb{P}^{r}}+D$ is linearly
equivalent to a positive multiple of a hyperplane section, and so
$X$ is of log-general type. We can therefore apply Proposition \ref{prop:Bundle-Cancel}
to deduce that for every isomorphism of abstract algebraic varieties
$\Phi:P\stackrel{\sim}{\rightarrow}Q$, there exists an automorphism
$\varphi$ of $X$ such that $P$ is isomorphic to $\varphi^{*}Q$
as a Zariski locally trivial $\mathbb{A}_{*}^{1}$-bundle over $X$.
In view of Corollary \ref{cor:IsoClass-logGeneral}, this means equivalently
that as a $\mathbb{G}_{m}$-bundle over $X$, $\varphi^{*}Q$ is isomorphic
to either $P$ or its inverse $P^{-1}$ in $H^{1}(X,\mathbb{G}_{m})$.
Since the choice of $k$ guarantees that the $\mathbb{G}_{m}$-bundle
$Q$ is isomorphic neither to $P$ nor to $P^{-1}$, the conclusion
follows from the observation that the natural action of $\mathrm{Aut}(X)$
on $H^{1}(X,\mathbb{G}_{m})$ is the trivial one, due to the fact
that every automorphism of $X$ is the restriction of a linear automorphism
of the ambient space $\mathbb{P}^{r}$. Indeed, through the open inclusion
$X\hookrightarrow\mathbb{P}^{r}$, we may consider an automorphism
$\varphi$ of $X$ as a birational self-map of $\mathbb{P}^{r}$ restricting
to an isomorphism outside $D$. If $\varphi$ is not biregular on
the whole $\mathbb{P}^{r}$, then $D$ would be an exceptional divisor
of $\varphi^{-1}$, in particular, $D$ would be birationally ruled,
in contradiction with the ampleness of its canonical divisor $K_{D}$
guaranteed by the condition $d\geq r+2$ .\end{proof}
\begin{example}
In the setting of Proposition \ref{thm:Non-Cancel-Gm} above, an isomorphism
$P\times\mathbb{A}_{*}^{1}\stackrel{\sim}{\rightarrow}Q\times\mathbb{A}_{*}^{1}$
can be constructed ``explicitly'' as follows. The complement $X\subset\mathbb{P}^{r}=\mathrm{Proj}(\mathbb{C}[x_{0},\ldots,x_{r}])$
of a smooth hypersurface $D$ defined by an equation $F(x_{0},\ldots,x_{r})=0$
for some homogeneous polynomial of degree $d$ can be identified with
the quotient of the smooth factorial affine variety $\tilde{X}\subset\mathbb{A}^{r+1}$
with equation $F(x_{0},\ldots,x_{r})=1$ by the free action of the
group $\mu_{d}=\mathrm{Spec}(\mathbb{C}[\varepsilon]/(\varepsilon^{d}-1))$
of $d$-th roots of unity defined by $\varepsilon\cdot(x_{0},\ldots,x_{r})=(\varepsilon x_{0},\ldots,\varepsilon x_{r})$.
The $\mathbb{G}_{m}$-bundles over $X$ corresponding to the line
bundles $\mathcal{O}_{\mathbb{P}^{r}}(k)\mid_{X}$, $k\in\mathbb{Z}$,
then coincide with the quotients of the trivial $\mathbb{A}_{*}^{1}$-bundles
$\tilde{X}\times\mathbb{A}_{*}^{1}=\tilde{X}\times\mathrm{Spec}(\mathbb{C}[t^{\pm1}])$
by the respective $\mu_{d}$-actions $\varepsilon\cdot(x_{0},\ldots,x_{n},t)=(\varepsilon x_{0},\ldots,\varepsilon x_{n},\varepsilon^{k}t)$,
$k\in\mathbb{Z}$. Now let $q:Q\rightarrow X$ be the $\mathbb{G}_{m}$-bundle
corresponding to $\mathcal{O}_{\mathbb{P}^{r}}(k)\mid_{X}$ for some
$k\in\{2,\ldots,d-2\}$ relatively prime with $d$, and let $a,b\in\mathbb{Z}$
be such that $ak-bd=1$. Then one checks that the following isomorphism
\begin{eqnarray*}
\tilde{\Phi}:\tilde{X}\times\mathbf{T}^{2}=\tilde{X}\times\mathrm{Spec}(\mathbb{C}[t_{1}^{\pm1},u_{1}^{\pm1}])\stackrel{\sim}{\rightarrow}\tilde{X}\times\mathbf{T}^{2}=\tilde{X}\times\mathrm{Spec}(\mathbb{C}[t_{2}^{\pm1},u_{2}^{\pm1}]), & (t_{1},u_{1})\mapsto(t_{2},u_{2})=(t_{1}^{k}u_{1},t_{1}^{bd}u_{2}^{a})
\end{eqnarray*}
of schemes over $\tilde{X}$ is equivariant for the actions, say $\mu_{d,1}$
and $\mu_{d,k}$, of $\mu_{d}$ defined respectively by $\varepsilon\cdot(x_{0},\ldots,x_{r},t_{1},u_{1})=(\varepsilon x_{0},\ldots,\varepsilon x_{r},\varepsilon t_{1},u_{1})$
and $\varepsilon\cdot(x_{0},\ldots,x_{r},t_{2},u_{2})=(\varepsilon x_{0},\ldots,\varepsilon x_{r},\varepsilon^{k}t_{2},u_{2})$.
It follows that $\tilde{\Phi}$ descends to an isomorphism 
\[
\Phi:(\tilde{X}\times\mathbf{T}^{2})/\mu_{d,1}\simeq P\times\mathbb{A}_{*}^{1}\stackrel{\sim}{\rightarrow}Q\times\mathbb{A}_{*}^{1}\simeq(\tilde{X}\times\mathbf{T}^{2})/\mu_{d,k}
\]
of schemes over $X\simeq\tilde{X}/\mu_{d}$. 
\end{example}

\subsection{\label{sub:Smooth-surfaces} Non-cancellation for smooth factorial
affine surfaces }

Since the Picard group of a smooth affine curve $C$ of log-general
type is either trivial if $C$ is rational or of positive dimension
otherwise, there no direct way to adapt the previous construction
using principal $\mathbb{G}_{m}$-bundles over algebraic curves to
produce $2$-dimensional candidate counter-examples for cancellation
by $\mathbb{A}_{*}^{1}$. Instead, we will use locally trivial $\mathbb{A}_{*}^{1}$-bundles
over certain orbifold curves $\tilde{C}$ which arise from suitably
chosen $\mathbb{A}_{*}^{1}$-fibrations $\pi:S\rightarrow C$ on smooth
affine surfaces $S$. 

\begin{parn} Let $S$ be a smooth affine surface $S$ equipped with
a flat fibration $\pi:S\rightarrow C$ over a smooth affine rational
curve $C$ whose fibers, closed or not, are all isomorphic to $\mathbb{A}_{*}^{1}$
over the corresponding residue fields when equipped with their reduced
structure%
\footnote{In particular, $\pi$ is an untwisted $\mathbb{A}_{*}^{1}$-fibration
in the sense of \cite{MiBook}.%
}. It follows from the description of degenerate fibers of $\mathbb{A}_{*}^{1}$-fibrations
given in \cite[Theorem 1.7.3]{MiBook} that $S$ admits a relative
completion into a $\mathbb{P}^{1}$-fibered surface $\overline{\pi}:\overline{S}\rightarrow C$
obtained from a trivial $\mathbb{P}^{1}$-bundle $\mathrm{pr}_{1}:C\times\mathbb{P}^{1}\rightarrow C$
with a fixed pair of disjoint sections $H_{0}$ and $H_{\infty}$
by performing finitely many sequences of blow-ups of the following
type: the first step consists of the blow-up of a closed point $c_{i}\in H_{0}$,
$i=1,\ldots,s$, with exceptional divisor $E_{1,i}$ followed by the
blow-up of the intersection point of $E_{1,i}$ with the proper transform
of the fiber $F_{i}=\mathrm{pr}_{1}^{-1}(\mathrm{pr}_{1}(c_{i}))$,
the next steps consist of the blow-up of an intersection point of
the last exceptional divisor produced with the proper transform of
the union of $F_{i}$ and the previous ones, in such a way that the
total transform of $F_{i}$ is a chain of proper rational curves with
the last exceptional divisors produced, say $E_{i,n_{i}}$, as the
unique irreducible component with self-intersection $-1$. The projection
$\mathrm{pr}_{1}:C\times\mathbb{P}^{1}\rightarrow C$ lifts on the
so-constructed surface $\overline{S}$ to a $\mathbb{P}^{1}$-fibration
$\overline{\pi}:\overline{S}\rightarrow C$ and $S$ isomorphic to
the complement of the union of the proper transforms of $H_{0}$ and
$H_{\infty}$ and of the divisors $F_{i}\cup E_{i,1}\cup\cdots\cup E_{i,n_{i}-1}$,
$i=1,\ldots,s$. The restriction of $\overline{\pi}$ to $S$ is indeed
an $\mathbb{A}_{*}^{1}$-fibration $\pi:S\rightarrow C$ with $s$
degenerate fibers $\pi^{-1}(\mathrm{pr}_{1}(c_{i}))$ isomorphic to
$E_{i,n_{i}}\cap S\simeq\mathbb{A}_{*}^{1}$ when equipped with their
reduced structure and whose respective multiplicities depend on the
sequences of blow-ups performed. 

\end{parn}

\begin{parn} \label{par:A1star-StackFactorFibers} It follows in
particular from this construction that $S$ admits a proper action
of the multiplicative group $\mathbb{G}_{m}$ which lifts the one
on $(C\times\mathbb{P}^{1})\setminus(H_{0}\cup H_{\infty})\simeq C\times\mathbb{A}_{*}^{1}$
by translations on the second factor. The local descriptions given
in \cite{FieK91} can then be re-interpreted for our purpose as the
fact that the $\mathbb{A}_{*}^{1}$-fibration $\pi:S\rightarrow C$
factors through an \'etale locally trivial $\mathbb{A}_{*}^{1}$-bundle
$\tilde{\pi}:S\rightarrow\tilde{C}$ over an orbifold curve $\delta:\tilde{C}\rightarrow C$
obtained from $C$ by replacing the finitely many points $c_{1},\ldots,c_{s}$
over the which the fiber $\pi^{-1}(c_{i})$ is multiple, say of multiplicity
$m_{i}>1$, by suitable orbifold points $\tilde{c}_{i}$ depending
only on the multiplicity $m_{i}$. More precisely, $\tilde{C}$ is
a smooth separated Deligne-Mumford stack of dimension $1$, of finite
type over $\mathbb{C}$ and with trivial generic stabilizer, which,
Zariski locally around $\delta^{-1}(c_{i})$ looks like the quotient
stack $[\tilde{U}_{c_{i}}/\mathbb{Z}_{m_{i}}]$, where $\tilde{U}_{c_{i}}\rightarrow U_{c_{i}}$
is a Galois cover of order $m_{i}$ of a Zariski open neighborhood
$U_{c_{i}}$ of $c_{i}$, totally ramified over $c_{i}$ and \'etale
elsewhere \cite{BeNo06}. 

\end{parn}
\begin{example}
\label{Ex:DM-quotient} Let $\mathbb{G}_{m}$ act on $\mathbb{A}_{*}^{2}=\mathrm{Spec}(\mathbb{C}[x,y])\setminus\{(0,0)\}$
by $t\cdot(x,y)=(t^{2}x,t^{5}y)$. The quotient $\mathbb{P}(2,5)=\mathbb{A}_{*}^{2}/\mathbb{G}_{m}$
is isomorphic to $\mathbb{P}^{1}$ and the quotient morphism $q:\mathbb{A}_{*}^{2}\rightarrow\mathbb{P}^{1}=\mathbb{A}_{*}^{2}/\mathbb{G}_{m}$
is an $\mathbb{A}_{*}^{1}$-fibration with two degenerate fibers $q^{-1}([0:1])$
and $q^{-1}([1:0])$ of multiplicities $5$ and $2$ respectively,
corresponding to the orbits of the points $(0,1)$ and $(1,0)$. In
contrast, the quotient stack $[\mathbb{A}_{*}^{2}/\mathbb{G}_{m}]$
is the Deligne-Mumford curve $\mathbb{P}[2,5]$ obtained from $\mathbb{P}^{1}$
by replacing the points $[0:1]$ and $[1:0]$ by ``stacky points''
with respective Zariski open neighborhoods isomorphic to the quotients
$[\mathbb{A}^{1}/\mathbb{Z}_{5}]$ and $[\mathbb{A}^{1}/\mathbb{Z}_{2}]$
for the actions of $\mathbb{Z}_{5}$ and $\mathbb{Z}_{2}$ on $\mathbb{A}^{1}=\mathrm{Spec}(\mathbb{C}[z])$
given by $z\mapsto\exp(2i\pi/5)z$ and $z\mapsto-z$. The quotient
morphism $q:\mathbb{A}_{*}^{2}\rightarrow\mathbb{P}^{1}$ factors
through the canonical morphism $\tilde{q}:\mathbb{A}_{*}^{2}\rightarrow\mathbb{P}[2,5]$
which is an \'etale local trivial $\mathbb{A}_{*}^{1}$-bundle, and
the induced morphism $\delta:\mathbb{P}[2,5]\rightarrow\mathbb{P}^{1}=\mathbb{A}_{*}^{2}/\mathbb{G}_{m}$
is an isomorphism over the complement of the points $[0:1]$ and $[1:0]$. 
\end{example}
In the next paragraphs, we construct two smooth affine surfaces $S_{1}$
and $S_{2}$ with an $\mathbb{A}_{*}^{1}$-fibration $\pi_{i}:S_{i}\rightarrow\mathbb{A}^{1}$,
$i=1,2$, factoring through a locally trivial $\mathbb{A}_{*}^{1}$-bundle
$\tilde{\pi}:S_{i}\rightarrow\mathbb{A}^{1}[2,5]$ over the affine
Deligne-Mumford curve $\mathbb{A}^{1}[2,5]$ obtained from the one
$\mathbb{P}[2,5]$ of Example \ref{Ex:DM-quotient} above by removing
a general scheme-like point.

\newpage

\begin{figure}[!htb]
\psset{linewidth=0.5pt}
\begin{pspicture}(-1.5,4.5)(9,-3.5)

\rput(-1,3){
\psscalebox{0.7}{
\psline[linestyle=dashed](0,0)(-3,-5)\psline[linecolor=gray](0,0)(3,-5)\psline[linestyle=dashed](-3,-5)(3,-5)

\pscurve[linestyle=dashed,linecolor=gray](1.6,-1.2)(1.1,-1.2)(0,0)
\pscurve[linestyle=dashed,linecolor=gray](0,0)(0.1,-1.1)(-0.5,-3.1)(-2.4,-4.9)(-3,-5)
\pscurve[linestyle=dashed,linecolor=gray](-3,-5)(-2.8,-5)(-2.1,-5.5)

\pscurve[linestyle=dashed,linecolor=gray](1.6,-0.8)(1,-0.8)(0,0)
\pscurve[linestyle=dashed,linecolor=gray](0,0)(-0.1,-1.1)(-0.8,-3.1)(-2.8,-4.9)(-3,-5)
\pscurve[linestyle=dashed,linecolor=gray](-3,-5)(-2.8,-5)(-2.1,-6.2)

\pscurve(1.5,-1)(1,-1)(0,0)
\pscurve(0,0)(0,-1)(-0.6,-3)(-2.5,-4.8)(-3,-5)
\pscurve(-3,-5)(-2.8,-5)(-2,-6)

\rput(-1,-1){$\mathbb{P}^2$}
\rput(0,0.2){$\infty$}
\rput(-3.2,-5){$0$}
\rput(-2.5,-3){$L_y$}
\rput(0,-5.4){$L_x$}
\rput(2.4,-3){$L_z$}
\rput(0,-3){$D_1$}
}
}

\rput(7,4){
\psscalebox{0.7 0.6}{
\psline(0,0.2)(0,-11.2)
\psline[linestyle=dashed,linecolor=gray](-0.8,0.2)(-0.8,-11.2)
\psline[linestyle=dashed,linecolor=gray](0.8,0.2)(0.8,-11.2)
\psline(-4,0)(4,0)\psline(-4,-11)(4,-11)

\psline(-3,0.5)(-4,-3)\psline(-4,-2)(-3,-5)\psline[linestyle=dashed](-3,-4)(-4,-7)\psline(-4,-6)(-3,-9)\psline(-3,-8)(-4,-11.5)
\psline(3,0.5)(4,-3.1)\psline[linecolor=gray](4,-2.5)(3,-5.8)\psline[linestyle=dashed](3,-5.2)(4,-8.5)\psline(4,-7.9)(3,-11.5)

\rput(4,1){$\tilde{S}_1$}
\rput(0.4,-4){$D_1$}\rput(-0.4,-4){$0$}
\rput(-1.5,0.4){$H_{\infty,1}=E_{\infty,4}$}\rput(-1.5,-0.3){$-1$}
\rput(-4,-1.2){$E_{\infty,3}$}\rput(-3.2,-1.5){$-2$}
\rput(-4,-3.8){$E_{\infty,1}$}\rput(-3.2,-3.4){$-3$}
\rput(-4,-5.2){$L_x$}\rput(-3.2,-5.6){$-1$}
\rput(-4,-8){$E_{0,1}$}\rput(-3.2,-7.4){$-2$}
\rput(-4,-9.8){$E_{0,2}$}\rput(-3.3,-10.2){$-3$}
\rput(-1.5,-11.5){$H_{0,1}=E_{0,4}$}\rput(-1.5,-10.7){$-1$}
\rput(4,-1.2){$E_{\infty,2}$}\rput(3,-1.2){$-3$}
\rput(4,-4.2){$L_z$}\rput(3,-4.2){$-1$}
\rput(4,-7){$L_y$}\rput(3,-7){$-2$}
\rput(4,-10){$E_{0,3}$}\rput(3,-10){$-2$}

\psline(-4,-12.2)(4,-12.2)
\rput(4.5,-12.2){$\mathbb{P}^1$}

\psline{->}(0,-11.5)(0,-12)\rput(0.5,-11.7){$\tilde{\xi}_1$}

}
}
\psline{->}(3,1)(1.5,1)
\end{pspicture}
\caption{Minimal resolution of the pencil $\xi_1:\mathbb{P}^2\dashrightarrow \mathbb{P}^1$. The exceptional divisors $E_{0,i}$ and $E_{\infty,i}$, $i=1,\ldots,4$,  over the respective proper base points $0=[0:0:1]$ and $\infty=[0:1:0]$ of $\xi_1$ are numbered according to the order of their extraction.}
\label{fig:reso1}
\end{figure}
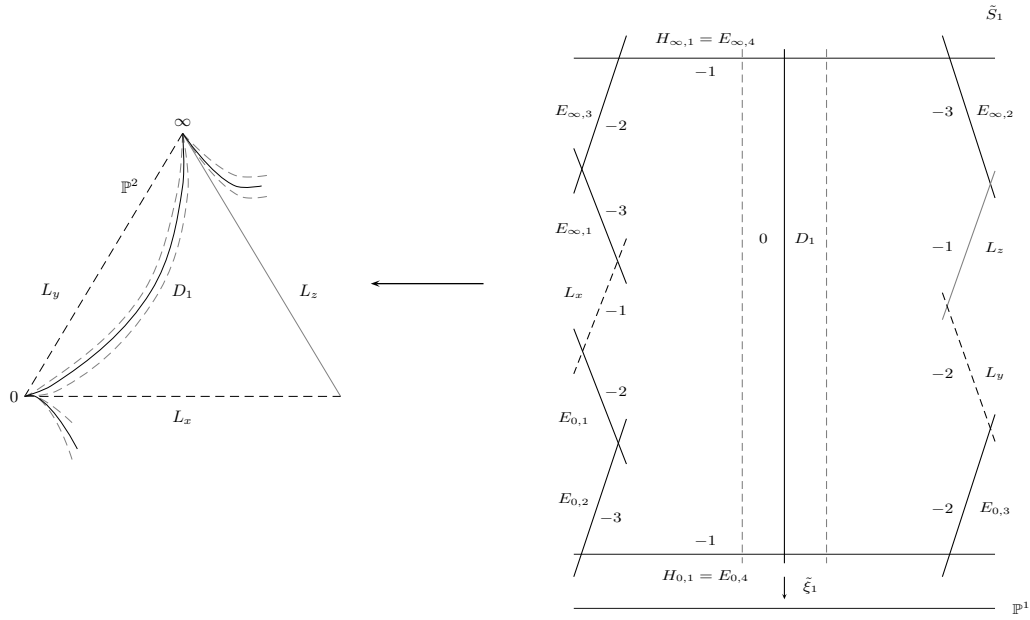

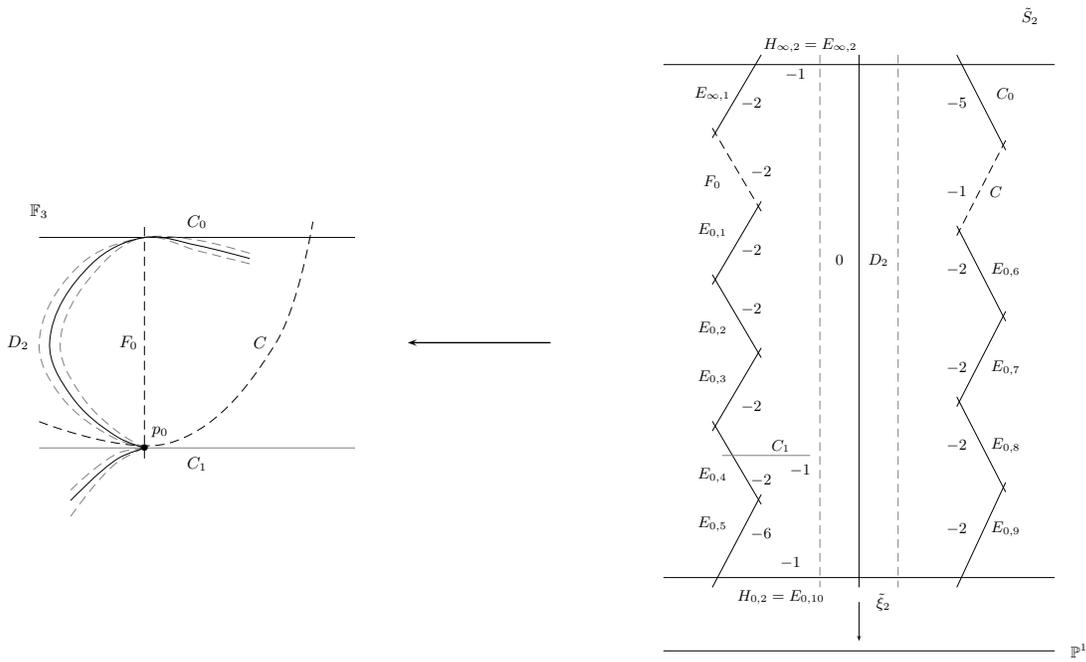
\begin{figure}[!htb]
\psset{linewidth=0.5pt}
\begin{pspicture}(-1.5,3.8)(9,-5)
\rput(-2,0.7){
\psscalebox{0.7}{
\psline(-2,0)(4,0)
\psline[linecolor=gray](-2,-4)(4,-4)
\pscurve[linestyle=dashed](-2,-3.5)(0.5,-3.9)(2.5,-2)(3.2,0.3)
\psline[linestyle=dashed](0,0.2)(0,-4.2)
\pscurve[linestyle=dashed,linecolor=gray](2,-0.3)(1,-0.05)(0,0)(-1.2,-0.5)(-2,-2)(-1.6,-3)(-0.6,-3.8)(0,-4)
\pscurve[linestyle=dashed,linecolor=gray](0,-4)(-0.5,-4.1)(-1.4,-4.8)

\pscurve[linestyle=dashed,linecolor=gray](2,-0.5)(1,-0.25)(0,0)(-0.8,-0.5)(-1.6,-2)(-1.2,-3)(-0.5,-3.7)(0,-4)
\pscurve[linestyle=dashed,linecolor=gray](0,-4)(-0.55,-4.3)(-1.4,-5.3)

\pscurve(2,-0.4)(1,-0.15)(0,0)(-1,-0.5)(-1.8,-2)(-1.4,-3)(-0.5,-3.8)(0,-4)
\pscurve(0,-4)(-0.5,-4.2)(-1.4,-5)

\rput(-2,0.5){$\mathbb{F}_3$}
\rput(1,0.3){$C_0$}
%%\rput(2,0.3){$-3$}
\rput(1,-4.3){$C_1$}
%%\rput(2,-4.3){$3$}
\rput(2.2,-2){$C$}
%%\rput(2.8,-2){$5$}
\rput(-0.3,-2){$F_0$}
%%\rput(0.3,-2){$0$}
\rput(-2.4,-2){$D_2$}
\rput(0,-4){\textbullet}\rput(0.3,-3.7){$p_0$}
}
}
\rput(7.5,3){
\psscalebox{0.65}{
\psline(-4,0)(4,0)
\psline(-4,-10.5)(4,-10.5)
\psline(0,0.2)(0,-10.7)
\psline[linestyle=dashed, linecolor=gray](0.8,0.2)(0.8,-10.7)
\psline[linestyle=dashed, linecolor=gray](-0.8,0.2)(-0.8,-10.7)
\psline(-2,0.2)(-3,-1.5)\psline[linestyle=dashed](-3,-1.3)(-2,-3)\psline(-2,-2.8)(-3,-4.5)\psline(-3,-4.3)(-2,-6)\psline(-2,-5.8)(-3,-7.5)\psline(-3,-7.3)(-2,-9)\psline(-2,-8.8)(-3,-10.7)
\psline[linecolor=gray](-2.8,-8)(-1,-8)
\psline(2,0.2)(3,-1.75)\psline[linestyle=dashed](3,-1.55)(2,-3.5)\psline(2,-3.3)(3,-5.25)\psline(3,-5.05)(2,-7)\psline(2,-6.8)(3,-8.75)\psline(3,-8.55)(2,-10.7)

\rput(3.5,1){$\tilde{S}_2$}
\rput(-1,0.4){$H_{\infty,2}=E_{\infty,2}$} \rput(-1.3,-0.2){$-1$}
\rput(-3,-0.6){$E_{\infty,1}$}\rput(-2.2,-0.8){$-2$}
\rput(-3,-2.4){$F_0$}\rput(-2,-2.2){$-2$}
\rput(-3,-3.4){$E_{0,1}$}\rput(-2.2,-3.8){$-2$}
\rput(-3,-5.4){$E_{0,2}$}\rput(-2.2,-5){$-2$}
\rput(-3,-6.4){$E_{0,3}$}\rput(-2.2,-7){$-2$}
\rput(-3,-8.4){$E_{0,4}$}\rput(-2,-8.5){$-2$}
\rput(-1.6,-7.8){$C_1$}\rput(-1.2,-8.3){$-1$}
\rput(-3,-9.4){$E_{0,5}$}\rput(-2,-9.6){$-6$}
\rput(-1.6,-10.9){$H_{0,2}=E_{0,10}$}\rput(-1.4,-10.2){$-1$}
\rput(3,-0.6){$C_0$}\rput(2,-0.8){$-5$}
\rput(2.8,-2.6){$C$}\rput(2,-2.6){$-1$}
\rput(3,-4.2){$E_{0,6}$}\rput(2,-4.2){$-2$}
\rput(3,-6.2){$E_{0,7}$}\rput(2,-6.2){$-2$}
\rput(3,-7.8){$E_{0,8}$}\rput(2,-7.8){$-2$}
\rput(3,-9.5){$E_{0,9}$}\rput(2,-9.5){$-2$}
\rput(0.4,-4){$D_2$}\rput(-0.4,-4){$0$}

\psline(-4,-12)(4,-12)
\rput(4.5,-12){$\mathbb{P}^1$}

\psline{->}(0,-11)(0,-11.8)\rput(0.5,-11){$\tilde{\xi}_2$}

}
}
\psline{->}(3.4,-0.7)(1.5,-0.7)
\end{pspicture}
\caption{Minimal resolution of the pencil $\xi_2:\mathbb{F}_3\dashrightarrow \mathbb{P}^1$. The exceptional divisors $E_{\infty,1}$, $E_{\infty,2}$ and $E_{0,i}$, $i=1,\ldots,10$,  over the respective proper base points $\infty=F_0\cap C_0$ and $0=p_0$ of $\xi_2$ are numbered according to the order of their extraction.}
\label{fig:reso2}
\end{figure}

\newpage

\begin{parn} The first one $S_{1}$ is equal to the complement in
$\mathbb{P}^{2}=\mathrm{Proj}(\mathbb{C}[x,y,z])$ of the union of
the cuspidal curve $D_{1}=\left\{ x^{5}-y^{2}z^{3}=0\right\} $ and
the line $L_{z}=\{z=0\}$. Equivalently, $S_{1}$ is the complement
in $\mathbb{A}^{2}=\mathrm{Spec}(\mathbb{C}[x,y])=\mathbb{P}^{2}\setminus L_{z}$
of the curve $D_{1}\cap\mathbb{A}^{2}=\{x^{5}-y^{2}=0\}$. This curve
being an orbit with trivial isotropy of the $\mathbb{G}_{m}$-action
$t\cdot(x,y)=(t^{2}x,t^{5}y)$ on $\mathbb{A}^{2}$, the composition
of the inclusion $S_{1}\hookrightarrow\mathbb{A}_{*}^{2}$ with the
canonical morphism $q:\mathbb{A}_{*}^{2}\rightarrow\mathbb{P}[2,5]=[\mathbb{A}_{*}^{2}/\mathbb{G}_{m}]$
defines a locally trivial $\mathbb{A}_{*}^{1}$-bundle $\tilde{\pi}_{1}:S_{1}\rightarrow\mathbb{P}[2,5]\setminus q(D_{1}\cap\mathbb{A}_{*}^{2})\simeq\mathbb{A}^{1}[2,5]$.
The rational pencil $\xi_{1}:\mathbb{P}^{2}\dashrightarrow\mathbb{P}^{1}$
induced by $\pi_{1}=\delta\circ\tilde{\pi}_{1}:S_{1}\rightarrow\mathbb{A}^{1}$
coincide with that generated by the pairwise linearly equivalent divisors
$D_{1}$, $5L_{x}$ and $3L_{z}+2L_{y}$, where $L_{x}$, $L_{y}$
and $L_{z}$ denote the lines $\{x=0\}$, $\{y=0\}$ and $\{z=0\}$
in $\mathbb{P}^{2}$ respectively. A minimal resolution $\tilde{\xi}_{1}:\tilde{S}_{1}\rightarrow\mathbb{P}^{1}$
of $\xi_{1}$ is depicted in Figure \ref{fig:reso1}. A relatively
minimal SNC completion $(\overline{S}_{1},B_{1})$ of $S_{1}$, with
boundary $B_{1}=D_{1}\cup H_{\infty,1}\cup H_{0,1}\cup E_{\infty,3}\cup E_{\infty,1}\cup E_{0,1}\cup E_{0,2}\cup E_{\infty,2}\cup E_{0,3}$,
on which $\pi_{1}$ extends to a $\mathbb{P}^{1}$-fibration $\overline{\pi}_{1}:\overline{S}_{1}\rightarrow\mathbb{P}^{1}$
is then obtained from $\tilde{S}_{1}$ by contracting the proper transform
of $L_{z}$. 

\end{parn}

\begin{parn} The second surface $S_{2}$ is obtained as follows.
In the Hirzebruch surface $\rho:\mathbb{F}_{3}=\mathbb{P}(\mathcal{O}_{\mathbb{P}^{1}}\oplus\mathcal{O}_{\mathbb{P}^{1}}(-3))\rightarrow\mathbb{P}^{1}$
with exceptional section $C_{0}$ of self-intersection $-3$, we choose
a section $C$ of $\rho$ in the linear system $|C_{0}+4F|$, where
$F$ denotes a general fiber of $\rho$, and a section $C_{1}$ in
the linear system $|C_{0}+3F|$ intersecting $C$ with multiplicity
$4$ in a unique point $p_{0}$. The fact that such pairs of sections
exists follows for instance from \cite[Lemma 3.2]{Dub14}. Let $\xi_{2}:\mathbb{F}_{3}\dashrightarrow\mathbb{P}^{1}$
be the pencil generated by the linearly equivalent divisors $C_{0}+5C$
and $6C_{1}+2F_{0}$ where $F_{0}=\overline{\rho}^{-1}(\rho(p_{0}))$.
Let $D_{2}$ be a general member of $\xi$ and let $S_{2}\subset\mathbb{F}_{3}$
be the complement of $C_{0}\cup C_{1}\cup D_{2}$. A minimal resolution
$\tilde{\xi}_{2}:\tilde{S}_{2}\rightarrow\mathbb{P}^{1}$ of $\xi_{2}:\mathbb{F}_{3}\dashrightarrow\mathbb{P}^{1}$
is depicted in Figure \ref{fig:reso2}. A relatively minimal SNC completion
$(\overline{S}_{2},B_{2})$ of $S_{2}$ with boundary $B_{2}=D_{2}\cup H_{\infty,2}\cup H_{0,2}\cup E_{\infty,1}\cup E_{0,5}\cup C_{0}\cup\bigcup_{i=6}^{9}E_{0,i}$,
on which $\pi_{2}$ extends to a $\mathbb{P}^{1}$-fibration $\overline{\pi}_{2}:\overline{S}_{2}\rightarrow\mathbb{P}^{1}$
is then obtained from $\tilde{S}_{2}$ by contracting successively
the proper transform of $C_{1}$, $E_{0,4}$, $E_{0,3}$, $E_{0,2}$
and $E_{0,1}$. By construction, the restriction of $\xi_{2}$ to
$S_{2}$ is an $\mathbb{A}_{*}^{1}$-fibration $\pi_{2}:S_{2}\rightarrow\mathbb{A}^{1}=\mathbb{P}^{1}\setminus\xi(D_{2})$
with two degenerate fibers: one of multiplicity $5$ supported by
$C\cap S_{2}\simeq\mathbb{A}_{*}^{1}$, and one of multiplicity $2$
supported by $F_{0}\cap S_{2}\simeq\mathbb{A}_{*}^{1}$. So by virtue
of $\S$ \ref{par:A1star-StackFactorFibers}, $\pi_{2}$ factors through
a locally trivial $\mathbb{A}_{*}^{1}$-bundle $\tilde{\pi}_{2}:S_{2}\rightarrow\mathbb{A}^{1}[2,5]$. 

\end{parn}
\begin{prop}
The surfaces $S_{1}$ and $S_{2}$ are smooth affine rational and
factorial. They are non isomorphic but $S_{1}\times\mathbb{A}_{*}^{1}$
is isomorphic to $S_{2}\times\mathbb{A}_{*}^{1}$. \end{prop}
\begin{proof}
Since $S_{1}$ is a principal open subset of $\mathbb{A}^{2}$, it
is smooth affine rational and factorial. The smoothness and the rationality
of $S_{2}$ are also clear. Since $D_{2}$ belongs to the linear system
$6C_{0}+20F$, it is ample by virtue of \cite[Theorem 2.17]{Ha77}.
This implies in turn that $C_{0}+C_{1}+D_{2}$ is the support of an
ample divisor, whence that $S_{2}$ is affine. Since the divisor class
group of $\mathbb{F}_{3}$ is generated by $C_{0}$ and $F$, the
identity $F\sim7C_{1}-C_{0}-D_{2}$ in the divisor class group of
$\mathbb{F}_{3}$ implies that every Weil divisor on $S_{2}$ is linearly
equivalent to one supported on the boundary $\mathbb{F}_{3}\setminus S_{2}=C_{0}\cup C_{1}\cup D_{2}$.
So $S_{2}$ is factorial. By construction, $S_{1}$ and $S_{2}$ both
have the structure of locally trivial $\mathbb{A}_{*}^{1}$-bundles
$\tilde{\pi}_{i}:S_{i}\rightarrow\mathbb{A}^{1}[2,5]$. The fiber
product $W=S_{1}\times_{\mathbb{A}^{1}[2,5]}S_{2}$ thus inherits
via the first and the second projection respectively the structure
of an \'etale locally trivial $\mathbb{A}_{*}^{1}$-bundle over $S_{1}$
and $S_{2}$. Since $H_{\textrm{ét}}^{1}(S_{i},\mathbb{G}_{m})\simeq H^{1}(S_{i},\mathbb{G}_{m})$
by virtue of Hilbert's Theorem 90 and $S_{i}$ is factorial, it follows
that $W$ is simultaneously isomorphic to the trivial $\mathbb{A}_{*}^{1}$-bundles
$S_{1}\times\mathbb{A}_{*}^{1}$ and $S_{2}\times\mathbb{A}_{*}^{1}$.
It remains to check that $S_{1}$ and $S_{2}$ are not isomorphic.
Suppose on the contrary that there exists an isomorphism $\varphi:S_{1}\stackrel{\sim}{\rightarrow}S_{2}$
and consider its natural extension as a birational map $\varphi:\overline{S}_{1}\dashrightarrow\overline{S}_{2}$
between the smooth SNC completions $(\overline{S}_{1},B_{1})$ and
$(\overline{S}_{2},B_{2})$ of $S_{1}$ and $S_{2}$ constructed above.
Then $\varphi$ must be a biregular isomorphism. Indeed, if either
$\varphi$ or $\varphi^{-1}$, say $\varphi$, is not regular then
we can consider a minimal resolution $\overline{S}_{1}\stackrel{\sigma}{\leftarrow}X\stackrel{\sigma'}{\rightarrow}\overline{S}_{2}$
of it. By definition of the minimal resolution, there is no $(-1)$-curve
in the union $B$ of the total transforms of $B_{1}$ and $B_{2}$
by $\sigma$ and $\sigma'$respectively which is exceptional for $\sigma$
and $\sigma'$ simultaneously, and $\sigma'$ consists of the contraction
of a sequence of successive $(-1)$-curves supported on $B$. The
only possible $(-1)$-curves in $B$ which are not exceptional for
$\sigma$ are the proper transforms of $D_{1}$ and of the two sections
$H_{0,1}$ and $H_{\infty,1}$ of the $\mathbb{P}^{1}$-fibration
$\overline{\pi}_{1}:\overline{S}_{1}\rightarrow\mathbb{P}^{1}$, but
the contraction of any of these would lead to a boundary which would
no longer be SNC, which is excluded by the fact that $B_{2}$ is SNC.
It follows that every isomorphism $\varphi:S_{1}\stackrel{\sim}{\rightarrow}S_{2}$
is the restriction of an isomorphism of pairs $(\overline{S}_{1},B_{1})\stackrel{\sim}{\rightarrow}(\overline{S}_{2},B_{2})$.
But no such isomorphism can exist due the fact that the intersection
forms of the boundaries $B_{1}$ and $B_{2}$ are different. Thus
$S_{1}$ and $S_{2}$ are not isomorphic, which completes the proof. 
\end{proof}

\section{complements and Open questions}

\subsection{Non-cancellation for higher dimensional tori}

Continuing the same idea as in section \ref{sec:non-Cancel} above,
it is possible to construct more generally pairs of principal homogeneous
$\mathbb{T}^{n}$-bundles over a given smooth variety $X$ whose total
spaces become isomorphic after taking their products with $\mathbf{T}^{n}$
but not with any other lower dimensional tori. For instance, one can
start with two collections $\{[p_{1}],\ldots,[p_{n}]\}$ and $\{[q_{1}],\ldots,[q_{n}]\}$
of classes in $H^{1}(X,\mathbb{G}_{m})$ which generate the same sub-group
$G$ of $H^{1}(X,\mathbb{G}_{m})$ and consider a pair of principal
homogeneous $\mathbb{T}^{n}$-bundles $p:P\rightarrow X$ and $q:Q\rightarrow X$
representing the classes $([p_{1}],\ldots,[p_{n}])$ and $([q_{1}],\ldots,[q_{n}])$
in $H^{1}(X,\mathbb{T})\simeq H^{1}(X,\mathbb{G}_{m})^{\oplus n}$.
Since $G$ is contained in the kernels of the natural homomorphism
$p^{*}:H^{1}(X,\mathbb{G}_{m})\rightarrow H^{1}(P,\mathbb{G}_{m})$
and $q^{*}:H^{1}(X,\mathbb{G}_{m})\rightarrow H^{1}(Q,\mathbb{G}_{m})$
(see Lemma \ref{lem:Bundle-Pic}), it follows that as a locally trivial
$\mathbb{\mathbf{T}}^{n}\times\mathbb{\mathbf{T}}^{n}$-bundle over
$X$, $P\times_{X}Q$ is simultaneously isomorphic to $P\times\mathbb{\mathbf{T}}^{n}$
and $Q\times\mathbb{\mathbf{T}}^{n}$. Then again, it remains to make
appropriate choices for $X$, $P$ and $Q$ which guarantee that for
every $n'=0,\ldots,n-1$, $P\times\mathbf{T}^{n'}$ and $Q\times\mathbf{T}^{n'}$
are not isomorphic as abstract algebraic varieties. 
\begin{thm}
\label{thm:Higher-Tori}Let $r\geq2$, let $d\geq r+2$ be a product
$n\geq1$ distinct prime numbers $5\leq d_{1}<\cdots<d_{n}$ and let
$X\subset\mathbb{P}^{r}$ be the complement of a smooth hypersurface
$D$ of degree $d$. Let $[p_{i}],[q_{i}]\in H^{1}(X,\mathbb{G}_{m})\simeq\mu_{d}$,
$i=1,\ldots,n$, be the classes corresponding via the isomorphism
$\mu_{d}\simeq\prod_{i=1}^{n}\mu_{d_{i}}$ to the elements $(1,\ldots,\exp(2\pi/d_{i}),\ldots,1)$
and $(1,\ldots,\exp(2\pi k_{i}/d_{i}),\ldots,1)$ for some $k_{i}\in\{2,\ldots,d_{i}-2\}$
and let $p:P\rightarrow X$ and $q:Q\rightarrow X$ be principal homogeneous
$\mathbb{T}^{n}$-bundles representing respectively the classes $([p_{1}],\ldots,[p_{n}])$
and $([q_{1}],\ldots,[q_{n}])$ in $H^{1}(X,\mathbb{T}^{n})$. 

Then for every $n'=0,\ldots,n-1$, the varieties $P\times\mathbf{T}^{n'}$
and $Q\times\mathbf{T}^{n'}$ are not isomorphic while $P\times\mathbf{T}^{n}$
and $Q\times\mathbf{T}^{n}$ are isomorphic as schemes over $X$. \end{thm}
\begin{proof}
Our choices guarantee that for every $0\leq n'<n$, the classes $([p_{1}],\ldots,[p_{n}],[1],\dots[1])$
and $([q_{1}],\ldots,[q_{n}],[1],\dots[1])$ in $H^{1}(X,\mathbb{T}^{n}\times\mathbb{T}^{n'})$
belong to distinct $\mathrm{GL}_{n+n'}(\mathbb{Z})$-orbits. Since
$d\geq r+2$, $X$ is of general type and $\mathrm{Aut}(X)$ acts
trivially on $H{}^{1}(X,\mathbb{G}_{m})$ (see the proof of Theorem
\ref{thm:Non-Cancel-Gm}). So the fact that $P\times\mathbf{T}^{n'}$
and $Q\times\mathbf{T}^{n'}$ are not isomorphic as abstract algebraic
varieties follows again from Corollary \ref{cor:IsoClass-logGeneral}.
On the other hand, since the classes $[p_{1}],\ldots,[p_{n}]$ and
$[q_{1}],\ldots,[q_{n}]$ both generate $H^{1}(X,\mathbb{G}_{m})$,
$P\times\mathbf{T}^{n}$ and $Q\times\mathbf{T}^{n}$ are isomorphic
$X$-schemes by virtue of the previous discussion. Alternatively,
one can observe that choosing $a_{i},b_{i}\in\mathbb{Z}$ such that
$a_{i}k_{i}+b_{i}d_{i}=1$ for every $i=1,\ldots,n$, the following
matrices $A$ and $B$ in $\mathrm{GL}_{2n}(\mathbb{Z})$ 
\[
A=\left(\begin{array}{cccccc}
1 & 0 & 0\\
0 & \ddots & 0 &  & 0_{n}\\
0 & 0 & 1\\
k_{1} & 0 & 0 & 1 & 0 & 0\\
0 & \ddots & 0 & 0 & \ddots & 0\\
0 & 0 & k_{n} & 0 & 0 & 1
\end{array}\right)\qquad B=\left(\begin{array}{cccccc}
a_{1} & 0 & 0 & 1 & 0 & 0\\
0 & \ddots & 0 & 0 & \ddots & 0\\
0 & 0 & a_{n} & 0 & 0 & 1\\
1 & 0 & 0\\
0 & \ddots & 0 &  & 0_{n}\\
0 & 0 & 1
\end{array}\right)
\]
map respectively the classes $([p_{1}],\ldots,[p_{n}],[1],\dots[1])$
and $([q_{1}],\dots[q_{n}],[1],\ldots,[1])$ onto the one $([p_{1}],\ldots,[p_{n}],[q_{1}],\dots[q_{n}])$,
providing isomorphisms $P\times\mathbf{T}^{n}\simeq P\times_{X}Q$
and $Q\times\mathbf{T}^{n}\simeq P\times_{X}Q$ of Zariski locally
trivial $\mathbf{T}^{2n}$-bundles over $X$. 
\end{proof}
The following Lemma relating the Picard group of the total space of
principal homogeneous $\mathbb{T}^{n}$-bundle with the Picard group
of its base is certainly well known. We include it here because of
the lack of appropriate reference. 
\begin{lem}
\label{lem:Bundle-Pic} Let $X$ be a normal variety, let $[p_{1}],\ldots,[p_{n}]$
be a collection of classes in $H^{1}(X,\mathbb{G}_{m})$, and let
$p:P\rightarrow X$ be the principal homogeneous $\mathbb{T}^{n}$-bundle
with class $([p_{1}],\ldots,[p_{n}])\in H^{1}(X,\mathbb{T}^{n})=H^{1}(X,\mathbb{G}_{m})^{\oplus n}$.
Then $H^{1}(P,\mathbb{G}_{m})\simeq H^{1}(X,\mathbb{G}_{m})/G$ where
$G=\langle[p_{1}],\ldots,[p_{n}]\rangle$ is the subgroup generated
by $[p_{1}],\ldots,[p_{n}]$. \end{lem}
\begin{proof}
The Picard sequence \cite{Ma75} for the fibration $p:P\rightarrow X$
reads 
\[
0\rightarrow H^{0}(X,\mathbb{\mathcal{U}}_{X})\rightarrow H^{0}(P,\mathbb{\mathcal{U}}_{E})\rightarrow H^{0}(\mathbb{G}_{m}^{n},\mathbb{\mathcal{U}}_{\mathbb{G}_{m}^{n}})\stackrel{\delta}{\rightarrow}H^{1}(X,\mathbb{G}_{m})\rightarrow H^{1}(P,\mathbb{G}_{m})\rightarrow H^{1}(\mathbb{G}_{m}^{n},\mathbb{G}_{m})=0
\]
where for a variety $Y$, $\mathcal{U}_{Y}$ denotes the sheaf cokernel
of the homomorphism $\mathbb{C}_{Y}^{*}\rightarrow\mathbb{G}_{m,Y}$
from the constant sheaf $\mathbb{C}^{*}$ on $Y$ to the sheaf $\mathbb{G}_{m,Y}$
of germs invertible functions on $Y$. We may choose a basis $(e_{1},\ldots,e_{n})$
of $H^{0}(\mathbb{G}_{m}^{n},\mathbb{\mathcal{U}}_{\mathbb{G}_{m}^{n}})\simeq\mathbb{Z}^{n}$
in such a way that the connecting homomorphism $\delta$ maps $e_{i}$
to $[g_{i}]$ for every $i=1,\ldots,n$. The assertion follows. 
\end{proof}

\subsection{Non Cancellation for smooth factorial affine varieties of low Kodaira
dimension ?}

Recall that by \cite[Theorem 3]{IiFu77}, cancellation for $\mathbf{T}^{n}$
holds over smooth affine varieties of log-general type. On the other
hand, since they arise as Zariski locally trivial $\mathbb{A}_{*}^{1}$-bundles
over varieties of log-general type, it follows from Iitaka \cite{IiFu77}
and Kawamata-Viehweg \cite{Kaw78} addition theorems that all the
counter-examples $X$ constructed in subsection \ref{sub:Torus-NonCancel-dim3}
have Kodaira dimension $\dim X-1\geq2$. Similarly, the examples constructed
in Theorem \ref{thm:Higher-Tori} as well as their products by low
dimensional tori have Kodaira dimension at least $2$. One can also
check directly that the two surfaces constructed in subsection \ref{sub:Smooth-surfaces}
have Kodaira dimension equal to $1$. This raises the question whether
cancellation holds for smooth factorial affine varieties of low Kodaira
dimension, in particular for varieties of negative Kodaira dimension.
The following Proposition shows that if counter-examples exist, they
must be at least of dimension $3$: 
\begin{prop}
Let $S$ and $S'$ be smooth factorial affine surfaces. If $S\times\mathbb{A}_{*}^{1}$
and $S'\times\mathbb{A}_{*}^{1}$ are isomorphic and $\kappa(S)$
(or, equivalently, $\kappa(S')$) is not equal to $1$, then $S$
and $S'$ are isomorphic.\end{prop}
\begin{proof}
In view of Iitaka-Fujita strong cancellation Theorem \cite{IiFu77},
we only have to consider the cases where $\kappa(S)=\kappa(S')=-\infty$
or $0$. In the first case, $S$ and $S'$ are isomorphic to products
of punctured smooth affine rational curves with $\mathbb{A}^{1}$
(see. e.g. \cite{GP12}) and so, the assertion follows from a combination
of the existing positive results for cancellation by $\mathbb{A}^{1}$
and $\mathbb{A}_{*}^{1}$. Namely, let $S=C\times\mathbb{A}^{1}$
and $S'=C'\times\mathbb{A}^{1}$, where $C$ and $C'$ are punctured
affine lines. If either $C$ or $C'$, say $C$, is not isomorphic
to $\mathbb{A}^{1}$, then $\kappa(C\times\mathbb{A}_{*}^{1})=\kappa(C)\geq0$
and so, by Iitaka-Fujita strong cancellation for $\mathbb{A}^{1}$,
every isomorphism between $S\times\mathbb{A}_{*}^{1}$ and $S'\times\mathbb{A}_{*}^{1}$
descends to an isomorphism between $C\times\mathbb{A}_{*}^{1}$ and
$C'\times\mathbb{A}_{*}^{1}$. Since cancellation by $\mathbb{A}_{*}^{1}$
holds for smooth affine curves, we deduce in turn that $C$ and $C'$
are isomorphic, whence that $S$ and $S'$ are isomorphic. 

In the case where $\kappa(S)=\kappa(S')=0$, we already observed that
cancellation holds if every invertible function on $S$ or $S'$ is
constant. Therefore, we may assume that $S$ and $S'$ both have non
constant units whence, by virtue of \cite[§5]{GP12}, belong up to
isomorphism to the following list of surfaces: $V_{0}=\mathbb{A}_{*}^{1}\times\mathbb{A}_{*}^{1}$,
and the complements $V_{k}$ in the Hirzebruch surfaces $\rho_{k}:\mathbb{F}_{k}\rightarrow\mathbb{P}^{1}$,
$k\geq1$, of a pair of sections $H_{0,k}$ and $H_{\infty,k}$ of
$\rho_{k}$ with self-intersection $k$ intersecting each others in
a unique point $p_{k}$, and a fiber $F$ or $\rho_{k}$ not passing
through $p_{k}$. All the surfaces $V_{k}$, $k\geq1$, admit an $\mathbb{A}_{*}^{1}$-fibration
$\pi_{k}:V_{k}\rightarrow\mathbb{A}_{*}^{1}=\mathrm{Spec}(\mathbb{C}[x^{\pm1}])$
induced by the restriction of the pencil on $\mathbb{F}_{k}$ generated
by $H_{0,k}$ and $H_{\infty,k}$. The unique degenerate fiber of
$\pi_{k}$, say $\pi_{k}^{-1}(1)$ up to a linear change of coordinate
on $\mathbb{A}_{*}^{1}$, is irreducible, consisting of the union
of the intersection with $V_{k}$ of the exceptional section $C_{0,k}$
of $\rho_{k}$, and of the fiber $F_{k}$ of $\rho_{k}$ passing through
$p_{k}$, counted with multiplicity $k$. Note that $V_{0}$ does
not contain any curve isomorphic to $\mathbb{A}^{1}$ whereas each
surface $V_{k}$, $k\geq1$, contains exactly two such curves: the
intersections $F_{k}\cap V_{k}$ and $C_{0,k}\cap V_{k}$. It follows
in particular that $V_{0}\times\mathbb{A}_{*}^{1}$ cannot be isomorphic
to any $V_{k}\times\mathbb{A}_{*}^{1}$, $k\geq1$. Now suppose that
there exists an isomorphism $\Phi:V_{k}\times\mathbb{A}_{*}^{1}\stackrel{\sim}{\rightarrow}V_{k'}\times\mathbb{A}_{*}^{1}$
for some $k,k'\geq1$. Since $\kappa((F_{k}\cap V_{k})\times\mathbb{A}_{*}^{1})=\kappa((C_{0,k}\cap V_{k})\times\mathbb{A}_{*}^{1})=-\infty$,
their respective images by $\Phi$ cannot be mapped dominantly on
$V_{k'}$ by the first projection and since $F_{k'}\cap V_{k'}$ and
$C_{0,k'}\cap V_{k'}$ are the unique curves isomorphic to $\mathbb{A}^{1}$
on $V_{k'}$, we conclude similarly as in the proof of Proposition
\ref{prop:Bundle-Cancel} that $\Phi$ map $(\pi_{k}^{-1}(1))_{\mathrm{red}}\times\mathbb{A}_{*}^{1}$
isomorphically onto $(\pi_{k'}^{-1}(1))_{\mathrm{red}}\times\mathbb{A}_{*}^{1}$.
This implies in turn that $\Phi$ restricts to an isomorphism between
the open subsets $U_{k}=\pi_{k}^{-1}(\mathbb{A}_{*}^{1}\setminus\left\{ 1\right\} )\times\mathbb{A}_{*}^{1}$
and $U_{k'}=\pi_{k'}^{-1}(\mathbb{A}_{*}^{1}\setminus\left\{ 1\right\} )\times\mathbb{A}_{*}^{1}$
of $V_{k}\times\mathbb{A}_{*}^{1}$ and $V_{k'}\times\mathbb{A}_{*}^{1}$
respectively. Now $\mathbb{A}_{*}^{1}\setminus\left\{ 1\right\} $
is of log-general tpe and since the restrictions of $\pi_{k}$ and
$\pi_{k'}$ to $U_{k}$ and $U_{k'}$ are trivial $\mathbb{A}_{*}^{1}$-bundles,
we deduce from Iitaka-Fujita strong cancellation Theorem \cite{IiFu77}
that the restriction of $\Phi$ to $U_{k}$ descends to an isomorphism
$\varphi:\mathbb{A}_{*}^{1}\setminus\left\{ 1\right\} \stackrel{\sim}{\rightarrow}\mathbb{A}_{*}^{1}\setminus\left\{ 1\right\} $
for which the following diagram commutes \[\xymatrix{U_k\times \mathbb{A}^1_* \simeq (\mathbb{A}^1_*\setminus\{1\})\times \mathbb{A}^1_* \times \mathbb{A}^1_* \ar[r]^{\Phi} \ar[d]_{\pi_k \circ \mathrm{pr}_1} & (\mathbb{A}^1_*\setminus\{1\})\times \mathbb{A}^1_* \times \mathbb{A}^1_* \simeq U_{k'} \times \mathbb{A}^1_* \ar[d]^{\pi_{k'}\circ \mathrm{pr}_1} \\ \mathbb{A}^1_*\setminus\{1\} \ar[r]^{\varphi} & \mathbb{A}^1_*\setminus\{1\}.}\] Summing
up, if it exists, an isomorphism $\Phi:V_{k}\times\mathbb{A}_{*}^{1}\stackrel{\sim}{\rightarrow}V_{k'}\times\mathbb{A}_{*}^{1}$
must be compatible with the $\mathbf{T}^{2}$-fibrations $\pi_{k}\circ\mathrm{pr}_{1}:V_{k}\times\mathbb{A}_{*}^{1}\rightarrow\mathbb{A}_{*}^{1}$
and $\pi_{k'}\circ\mathrm{pr}_{1}:V_{k'}\times\mathbb{A}_{*}^{1}\rightarrow\mathbb{A}_{*}^{1}$.
But this impossible since the mutliplicity of the irreducible component
$F_{k}\cap V_{k}$ of $\pi_{k}^{-1}(1)$ is different for every $k\geq1$.
In conclusion, the surfaces $V_{k}$, $k\geq0$, are pairwise non
isomorphic, with pairwise non isomorphic $\mathbb{A}_{*}^{1}$-cylinders,
which completes the proof.  
\end{proof}
\bibliographystyle{amsplain}

\end{document}